\documentclass[12]{article}
\usepackage{setspace}
\usepackage{tikz}
\usepackage[margin=1in]{geometry}
\usepackage{amssymb}
\usepackage{amsfonts}
\usepackage{amsmath}
\usepackage{fullpage}
\usepackage{graphicx}
\usepackage[utf8]{inputenc}
\usepackage{amsthm}
\usepackage{textcomp}
\usepackage{float}
\usepackage{bm}
\usepackage{xcolor}
\usepackage{enumerate}
\usepackage{enumitem}
\usepackage{hyperref}
\usepackage[normalem]{ulem}
\usepackage{comment}

\newtheorem*{theoremA}{Theorem A}

{\par\hfill$\Box$}

\newtheorem{theorem}{Theorem} 
\newtheorem{corollary}{Corollary} 
\newtheorem{lemma}{Lemma} 

\counterwithout{equation}{section}

\newcommand{\rE}{\mathrm{E}} 
\newcommand{\rV}{\mathrm{Var}} 
\newcommand{\rC}{\mathrm{Cov}} 

\begin{document}

\title{The Multinomial Allocation Model and the Random Box Load}

\author{
Serik Sagitov%
\thanks{Chalmers University of Technology and the University of Gothenburg, Sweden. Email: \texttt{serik@chalmers.se}}
}

\maketitle

\begin{abstract}
We revisit the random allocation model in which $n$ balls are independently placed into $N$ boxes with probabilities $q_1,\ldots,q_N$. A classical asymptotic result due to Kolchin, Sevastyanov, and Chistyakov for the expectations, variances, and covariances of the occupancy counts is reformulated in a compact and transparent  form in terms of the load of a randomly selected box. We further derive explicit two-sided bounds for the associated remainder terms, obtained under weaker assumptions than those previously required.
\end{abstract}


\section{Introduction}

The classical occupancy problem is a cornerstone of probability theory and combinatorics \cite{F1}, with deep connections to statistics, computer science, and population biology. In its simplest form, $n$ balls are independently placed into $N$ boxes, each selected with equal probability $1/N$. Fundamental quantities of interest include the number of empty boxes $\hat N_0$, the number $\hat N_r$ of boxes containing exactly $r$ balls, and the maximum occupancy; see \cite{GJM} for a recent treatment of the latter.

A natural generalization replaces the equiprobable allocation by  the \textit{multinomial allocation model} \cite{KSC, BHJ}. In this setting, each box $k$ is chosen with its own probability $q_k$, where
\[
q_1 \ge q_2 \ge \cdots \ge q_N > 0, \qquad q_1 + \cdots + q_N = 1.
\]
Thus, the vector $(q_1,\ldots,q_N)$ describes how the load is distributed across boxes, allowing for heterogeneous box weights.

In this paper, we revisit Theorem~5 in \cite[Ch.~III.1]{KSC}, which analyzes the multinomial allocation model in the so-called \textit{central region}. To make the result more transparent, we restate it as Theorem~A in terms of the empirical proportions
\[
\hat q_r := \frac{\hat N_r}{N}, \qquad r = 0,1,\ldots,n.
\]
Here, $\hat q_r$ is the proportion of boxes containing exactly $r$ balls; in particular, $\hat q_0$ is the proportion of empty boxes. These quantities form a probability distribution on $\{0,\ldots,n\}$, since
\[
\sum_{r=0}^{n} \hat q_r = 1,
\]
and their first moment satisfies
\[
\sum_{r=1}^{n} r \hat q_r = \alpha, \qquad \alpha := \frac{n}{N},
\]
where $\alpha$ is the average number of balls per box (the \textit{average load}).

Our reformulation is based on a simple randomization device. Let
\[
\xi = n q_X, \qquad X \sim \mathrm{Uniform}\{1,\ldots,N\},
\]
so that $X$ selects a box uniformly at random. The random variable $\xi$ can be interpreted as the load of a randomly chosen box. By construction,
\[
\rE[\xi] = \alpha.
\]
Moreover,
\[
0 < \alpha \le \beta, \qquad \beta := n q_1,
\]
where $\beta$ is the maximal possible value of $\xi$. The case $\alpha = \beta$ characterizes the equiprobable allocation model, in which $\xi \equiv \alpha$ is deterministic.

\begin{theoremA}[\cite{KSC}]
Assume that, as $n \to \infty$, there exist positive constants $C_1$ and $C_2$, independent of $n$, such that
\begin{equation}\label{alpha}
C_1 \le \alpha \le \beta \le C_2.
\end{equation}
For all non-negative integers $r$, define
\[
p_{r}(x) := \frac{x^r}{r!} e^{-x}.
\]
Then, for any fixed $r \ne t$,
\begin{align*}
\rE[\hat q_r] &= \rE[p_{r}(\xi)] + O(n^{-1}),\\
\rV[\hat q_r] &= n^{-1}\alpha \rE\!\left[p_{r}(\xi)\bigl(1 - p_{r}(\xi)\bigr)\right]
               - n^{-1}\!\left(\rE\!\left[p_{r}(\xi)(\xi - r)\right]\right)^2
               + O(n^{-2}),\\
\rC[\hat q_r,\hat q_t] &= -n^{-1}\alpha \rE[p_{r}(\xi)p_{t}(\xi)]
                         - n^{-1}\rE[p_{r}(\xi)(\xi - r)]\,
                           \rE[p_{t}(\xi)(\xi - t)]
                         + O(n^{-2}).
\end{align*}
\end{theoremA}

Our main result, presented in the next section, extends the approximations in Theorem~A by allowing the parameters of the model $(N, q_1, \ldots, q_N)$ to depend on $n$, without imposing the restriction~\eqref{alpha}.

\section{Main result}\label{Mn}

\begin{theorem} \label{tm1}
For $0\le r,t\le n$ and $r\ne t$,  the following three approximation formulas hold
\begin{align}
\rE[\hat q_r]&=\rE[p_{r}(\xi)]+(2n)^{-1}\rE[p_{r}(\xi)(r-(\xi-r)^2)]+n^{-2}R_1(n,r), \label{R1}\\
\rV[\hat q_r]&=n^{-1}\alpha\rE[p_{r}(\xi)(1-p_{r}(\xi))]-n^{-1}(\rE[p_{r}(\xi)(\xi-r)])^2+n^{-2}(R_2(n,r,r)+\alpha R_1(n,r)), \label{R2}\\
\rC[\hat q_r,\hat q_t]&=-n^{-1}\alpha \rE[p_{r}(\xi)p_{t}(\xi)]-n^{-1}\rE[p_{r}(\xi)(\xi-r)]\,\rE[p_{t}(\xi)(\xi-t)]+n^{-2}R_2(n,r,t), \label{R3}
\end{align}
where provided  $q_1\le1/4$, 
\begin{align}
- (r!)^{-1} \beta^r(r^32^{r-1}+4) &\le R_1(n,r) \le (r!)^{-1}\beta^rr^2(2r^2+ 4\beta^2),  \label{D1}\\
- (r!t!)^{-1}\beta^{2u}L_2(u,\beta) &\le R_2(n,r,t) \le (r!t!)^{-1}\beta^{2u} K_2(u,\beta), \label{D3}
\end{align}
with $u=\max(r,t)$ and
\begin{align*}
L_2(r,\beta)&=12+2\beta+5(r^3+\beta)2^{r}+\beta^2(r^2+1),\\
K_2(r,\beta)&=8+12r^3 2^{r}+4\beta r+  \beta^2 2^{2r+6} r^2.
\end{align*}
\end{theorem}

\begin{corollary}\label{coro}
For the proportion of empty boxes $\hat q_0$, we get
 \begin{align*}
\rE[\hat q_0]&=\rE[e^{-\xi}]-(2n)^{-1}\rE[\xi^2 e^{-\xi}]+n^{-2}R_1(n),\qquad\qquad\qquad\qquad    \\
\rV[\hat q_0]&=n^{-1}\alpha\rE[e^{-\xi}(1-e^{-\xi})]-n^{-1}(\rE[\xi e^{-\xi}])^2+n^{-2}R_2(n),
\end{align*}
with
\[-4\le R_1(n) \le 0,\qquad -\beta^2-7\beta-4\alpha -12\le R_2(n) \le 8.\]
\end{corollary}

\begin{figure}[ht]
  \centering
  \includegraphics[height=5.1cm]{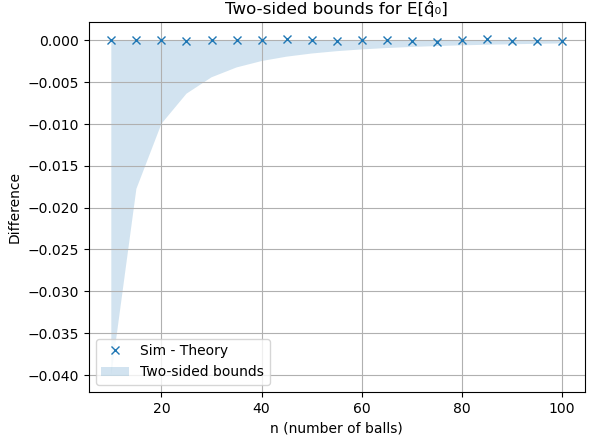} \includegraphics[height=5.2cm]{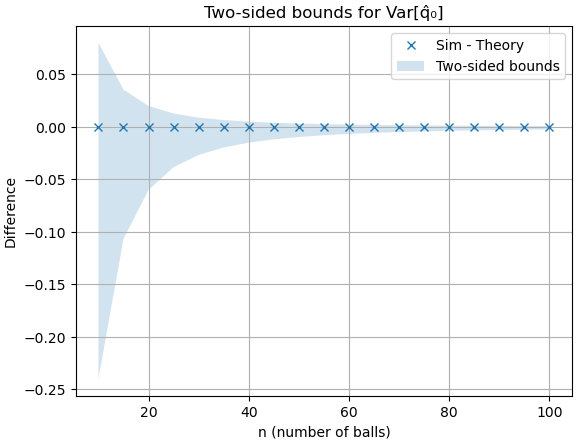}
\caption{Simulation results illustrating the bounds~\eqref{Ex1}--\eqref{Ex2} for $N = 100$ and $n = 10, 15, \ldots, 100$.}
  \label{fig1}
\end{figure}

\vspace{2mm}
\noindent\textbf{Example.}
Turning to the equiprobable allocation, we obtain
\begin{align}
-4n^{-2} \le\;& \rE[\hat q_0] - e^{-\alpha} + \tfrac{1}{2}n^{-1}\alpha^2 e^{-\alpha} \le 0, \label{Ex1}\\
-(\alpha^2 + 11\alpha + 12)n^{-2} \le\;& \rV[\hat q_0] - n^{-1}\bigl(\alpha e^{-\alpha} - \alpha e^{-2\alpha} - \alpha^2 e^{-2\alpha}\bigr) \le 8n^{-2}. \label{Ex2}
\end{align}
Relation~\eqref{Ex1} is informative provided that the higher-order term $4n^{-2}$ is smaller than $\tfrac{1}{2}n^{-1}\alpha^2 e^{-\alpha}$, that is,
\[
n > \tfrac{1}{8}\alpha^{-2} e^{\alpha}.
\]
Thus the two-sided bounds \eqref{Ex1}  remain useful beyond the central region, when $C_1 \le \alpha \le C_2$.

Figure~\ref{fig1} illustrates the bounds~\eqref{Ex1}--\eqref{Ex2} for $N = 100$ and $n = 10, 15, \ldots, 100$. For each $n$, the quantities $\rE[\hat q_0]$ and $\rV[\hat q_0]$ are estimated from $50{,}000$ independent simulations of the equiprobable allocation model. The plots show the differences between the simulated means and variances and their corresponding theoretical approximations (up to terms of order $n^{-1}$), together with the associated two-sided bounds of order $n^{-2}$.

\section{Auxiliary results}\label{kl}

\begin{lemma}\label{tm0}
Define
\[\Phi_n(r,t)=\binom{n}{r,t,n-r-t}
\Big(
 \rE\!\left[q_X^r q_Y^t (1-q_X-q_Y)_+^{\,n-r-t}\right]
- N^{-1} \rE\!\left[q_X^{r+t}(1-2q_X)_+^{\,n-r-t}\right]\Big)-\rE[\hat q_r]\rE[\hat q_t],\]
where $q_Y$ is an independent copy of $q_X$, and $(a)_+ := \max(a,0)$. 
Then for all $r \ne t$, 
\begin{align*}
\rE[\hat q_r]
&= \binom{n}{r}\, \rE\!\left[q_X^r (1-q_X)^{n-r}\right],\qquad
\rV[\hat q_r]
= \Phi_n(r,r)+N^{-1}\rE[\hat q_r],\qquad
\rC[\hat q_r ,\hat q_t]= \Phi_n(r,t).
\end{align*}
\end{lemma}

\begin{proof}
According to formulas (8), (9), and (10) in \cite[Ch.~III.1]{KSC},
\begin{align*}
\rE[\hat N_r]
&= \binom{n}{r}\sum_{k=1}^N q_k^r(1-q_k)^{\,n-r},\\
\rE[\hat N_r(\hat N_r-1)]
&= \binom{n}{r,r,n-2r}
   \sum_{k\ne l} q_k^r q_l^r(1-q_k-q_l)^{\,n-2r},\\
\rE[\hat N_r \hat N_t]
&= \binom{n}{r,t,n-r-t}
   \sum_{k\ne l} q_k^r q_l^t(1-q_k-q_l)^{\,n-r-t},\quad r \ne t.
\end{align*}
Observe that for any functions $F(\cdot)$ and $G(\cdot,\cdot)$,
\begin{align*}
N^{-1}\sum_{k=1}^N F(q_k) &= \rE[F(q_X)], \\
N^{-2}\sum_{k,l=1}^N G(q_k,q_l) &= \rE[G(q_X,q_Y)].
\end{align*}
Hence,
\[
N^{-2}\sum_{k\ne l} G(q_k,q_l)
= \rE[G(q_X,q_Y)] - N^{-1} \rE[G(q_X,q_Y)],
\]
which yields the stated expressions and completes the proof.
\end{proof}

\noindent\textbf{Remark.} Note that if $q_1 \le \tfrac{1}{2}$, then
\[
(1-q_X-q_Y)_+ = 1-q_X-q_Y, \qquad (1-2q_X)_+ = 1-2q_X.
\]
We impose the stronger condition $q_1 \le \tfrac{1}{4}$ in order to apply Lemma~\ref{L5} below.

\begin{lemma}\label{L4}
 For $0\le r\le n$, $n\ge1$, define $\varphi_n(\cdot)$  by
\[ \frac{n!}{(n-r)!}=n^r (1-n^{-1}\varphi_n(r)).\]
Then
\[ 0\le \varphi_n(r)\le \binom{r}{2},\]
and moreover, $\varphi^*_n(r)$ defined by
\[ \varphi_n(r)=\binom{r}{2}-n^{-1}\varphi^*_n(r),\]
satisfies
\[  0\le \varphi^*_n(r)\le r^4,\quad 0\le r\le n.\]
\end{lemma}

\begin{proof} 
Since
 \[\varphi_n(r)=\frac{n^{r}-n(n-1)\cdots(n-r+1)}{n^{r-1}}=\sum_{k=1}^{r-1} kn^{1-k}(n-1)\cdots(n-k+1), \]
we get
 \[0\le \varphi_n(r)\le \sum_{k=1}^{r-1} k=\binom{r}{2}.\]
 It remains to observe that
\begin{align*}
n^{-1}\varphi^*_n(r)&=\binom{r}{2}-\sum_{k=1}^{r-1} kn^{1-k}(n-1)\cdots(n-k+1)\\
&=\sum_{k=1}^{r-1} k\Big(1-\big(1-\tfrac{1}{n}\big)\cdots\big(1-\tfrac{k-1}{n}\big)\Big)\le n^{-1} \sum_{k=1}^{r-1} k^3\le n^{-1}r^4.
\end{align*}

 \end{proof}

\begin{lemma}\label{L5}
For $n \ge 1$ and $0\le r\le n$, define $\delta_n(r,x)$ by
\[
\left(1-\frac{x}{n}\right)^{n-r} = e^{-x} + n^{-1}\delta_n(r,x).
\]
If $0 \le x \le \frac{n}{2}$, then
\begin{equation}\label{del}
 -2 \le \delta_n(r,x) \le  r2^{r-1},
\end{equation}
and moreover, $\delta^*_n(r,x)$ defined by
\[\delta_n(r,x)=e^{-x}x(r-x/2)+n^{-1}\delta^*_n(r,x),
\]
satisfies 
\begin{equation}\label{del*}
-4-r\le \delta^*_n(r,x)\le x^2r(r+1) 2^{\,r+1}.
\end{equation}
\end{lemma}

\begin{proof}
We use the representation
\[
\left(1-\frac{x}{n}\right)^{n-r}
= \left(1-\frac{x}{n}\right)^{-r} e^{\,n\ln(1-x/n)}
\]
and repeatedly apply Taylor's theorem with Lagrange remainder. By
\[
\ln(1-y) = -y - \frac{y^2}{2(1-\theta y)^2}, \qquad \theta \in [0,1],
\]
we have
\[
-y - 2y^2\le \ln(1-y)\le -y,\quad 0\le y\le 1/2.
\]
With $y = \frac{x}{n}$, it follows that
\[
-x - \frac{2x^2}{n}
\le n\ln\left(1-\frac{x}{n}\right)
\le -x,\quad 0\le x\le n/2.
\]
Exponentiating and using the inequality $e^{-t} \ge 1 - t$ for $t \ge 0$, we obtain
\[
e^{-x}\left(1 - \frac{2x^2}{n}\right)
\le e^{\,n\ln(1-x/n)}
\le e^{-x}.
\]

On the other hand,\[
1 \le \left(1-\frac{x}{n}\right)^{-r}
\le 1 + \frac{xr}{n}2^{\,r}.
\]
Combining the bounds, we find
\[
-2x^2 e^{-x}
\le \delta_n(r,x)
\le e^{-x} x r 2^{\,r},
\]
which  yields \eqref{del} in view of $e^{-x}x<1/2$ and $e^{-x}x^2<1$.

The higher order Taylor's expansion,
\[
\ln(1-y) = -y - \frac{y^2}{2}- \frac{y^3}{3(1-\theta y)^3}, \qquad \theta \in [0,1],
\]
yields
\[
-x - \frac{x^2}{2n}- \frac{8x^3}{3n^2}
\le n\ln\left(1-\frac{x}{n}\right)
\le -x- \frac{x^2}{2n},\quad 0\le x\le n/2.
\]
Exponentiating and using 
\[e^{-t} = 1 - t+\frac{t^2}{2}e^{-\theta't},\quad t \ge 0, \qquad \theta' \in [0,1],\]
we obtain
\[
e^{-x}\left(1 - \frac{x^2}{2n}- \frac{8x^3}{3n^2}\right)
\le e^{\,n\ln(1-x/n)}
\le e^{-x}\left(1 - \frac{x^2}{2n}\right)+\frac{x^4}{4n^2}.
\]

On the other hand,
\[
 \left(1-\frac{x}{n}\right)^{-r}
= 1 + \frac{xr}{n}+\frac{r(1+r)x^2}{2n^2}\left(1-\frac{\theta''x}{n}\right)^{-r-2}, \qquad \theta'' \in [0,1],
\]
yields
\[
1 + \frac{rx}{n}\le \left(1-\frac{x}{n}\right)^{-r}
\le 1 + \frac{rx}{n}+\frac{r(1+r)x^2}{n^2}2^{r+1},
\]
Combining these bounds, and using $e^{-x}x^3<3/2$, we establish \eqref{del*}.
\end{proof}

\begin{lemma}\label{Co2}
Denote
\[Q_n(r,t):=(r!t!)^{-1}\rE\Big[\xi^r\eta^t(1-q_X-q_Y)^{n-r-t}\Big],\]
where $\eta=nq_Y$.
Then
\begin{equation}\label{Q1}
Q_n(r,t)=\rE[p_r(\xi)]\rE[p_t(\xi)]+n^{-1}(r! t!)^{-1}\rE\Big[\xi^r\eta^t\delta_n(r+t,\xi+\eta)\Big],
\end{equation}
 and 
\begin{align}\label{Q2}
 Q_n(r,t)&=\rE[p_r(\xi)]\rE[p_t(\xi)]+n^{-1}\Big((r+t)\big(\rE[p_r(\xi)]\rE[p_t(\xi)\xi]+\rE[p_t(\xi)]\rE[p_r(\xi)\xi]\big)-\rE[p_r(\xi)\xi]\rE[p_t(\xi)\xi]\Big) \nonumber\\
&-(2n)^{-1}\Big(\rE[p_r(\xi)]\rE[p_t(\xi)\xi^2]+\rE[p_t(\xi)]\rE[p_r(\xi)\xi^2]\Big)+n^{-2}(r!t!)^{-1}\rE\Big[\xi^r\eta^r \delta_n^*(r+t,\xi+\eta)\Big]. 
\end{align}
 \end{lemma}
\begin{proof}
 Relation \eqref{Q1} follows from
 \[(1-q_X-q_Y)^{n-r-t}=e^{-\xi-\eta}+n^{-1}\delta_n(r+t,\xi+\eta).\]
On the other hand, in terms of the function $\delta^*_n(r,x)$, we have
\begin{align*}
 Q_n(r,t)&=\rE[p_r(\xi)]\rE[p_t(\xi)]+n^{-1}\rE[p_r(\xi)p_t(\eta)(\xi+\eta)(r+t-2^{-1}(\xi+\eta))]\\
&+n^{-2}(r!t!)^{-1}\rE\Big[\xi^r\eta^r \delta_n^*(r+t,\xi+\eta)\Big],
\end{align*}
leading to \eqref{Q2}.
\end{proof}

\begin{lemma}\label{Ld}
Assuming $q_1\le1/4$, define $R_0(n,r)$ by
\begin{align*}
\rE[\hat q_r]&=\rE[p_{r}(\xi)]+n^{-1}R_0(n,r).
\end{align*}
Then for $0\le r\le n$ and $n\ge 1$,
\[-(r!)^{-1}\beta ^{r}(r^2+2)\le  R_0(n,r)\le (r!)^{-1}r(2\beta)^{r}.\]
\end{lemma}

\begin{proof}
 
Using Lemma \ref{tm0}, and the definition of $\varphi_n$ in Lemma \ref{L4}, we obtain
\begin{align*}
 \rE[\hat q_r]&={n\choose r}\rE[q_X^r(1-q_X)^{n-r}]=(r!)^{-1}\Big(1-n^{-1}\varphi_n(r)\Big)\rE[\xi^r(1-n^{-1}\xi)^{n-r}].
\end{align*}
Further, using  Lemma \ref{L5}, we get
\begin{align*}
 \rE[\hat q_r]&=\rE[p_{r}(\xi)]+(r!n)^{-1}\rE[\xi^r\delta_n(r,\xi))]-(r!n)^{-1}\varphi_n(r)\rE[\xi^r(1-n^{-1}\xi)^{n-r}].
\end{align*}
Thus 
\[r! \,R_0(n,r)=\rE[\xi^r\delta_n(r,\xi)]-\varphi_n(r)\rE[\xi^r(1-n^{-1}\xi)^{n-r}],\]
and the stated upper and lower bounds  follow from Lemmas  \ref{L4} and \ref{L5}.
\end{proof}

\section{Proof of Theorem \ref{tm1}}\label{PR2}

\subsection*{Proof of  (\ref{R1}) and (\ref{D1})}

In terms of Lemma \ref{L4}, Lemma \ref{tm0} gives
\begin{align*}
 \rE[\hat q_r]&=(r!)^{-1}\Big(1-n^{-1}\binom{r}{2}+n^{-2}\varphi^*_n(r)\Big)\rE[\xi^r(1-n^{-1}\xi)^{n-r}].
\end{align*}
By  Lemma \ref{L5},
\begin{align*}
(r!)^{-1}\rE[\xi^r(1-n^{-1}\xi)^{n-r}]&=\rE[p_{r}(\xi)]+n^{-1}(r!)^{-1}\rE[\xi^r\delta_n(r,\xi))]\\
&=\rE[p_{r}(\xi)(1+n^{-1}\xi(r-\xi/2))]+n^{-2}(r!)^{-1}\rE[\xi^r\delta^*_n(r,\xi))],
\end{align*}
which yields
\begin{align*}
 \rE[\hat q_r]&=\rE[p_{r}(\xi)]+n^{-1}J+n^{-2} R_1(n,r),
\end{align*}
where 
\begin{align*}
J&=\rE[p_{r}(\xi)(r-\xi/2)\xi]-\binom{r}{2}\rE[p_{r}(\xi)]=\frac{1}{2}\rE[p_{r}(\xi)(r-(\xi-r)^2)].
\end{align*}
Thus \eqref{R1} holds with
\begin{align*}
r!\, R_1(n,r)&=\rE[\xi^r\delta^*_n(r,\xi)]-\binom{r}{2}\rE[\xi^r\delta_n(r,\xi)]+\varphi^*_n(r)\rE[\xi^r(1-n^{-1}\xi)^{n-r}].
\end{align*}
It remains to apply the two-sided bounds of Lemma \ref{L4} and  Lemma \ref{L5} to derive \eqref{D1}.


\subsection*{Proof of  (\ref{R3}) and (\ref{D3})}
In view of Lemma~\ref{tm0}, it suffices to show that
\begin{align}\label{sist}
n\Phi_n(r,t)
&= -\alpha \rE[p_{r}(\xi)p_{t}(\xi)]
   - \rE[p_{r}(\xi)\xi]\,\rE[p_{t}(\xi)\xi]    + r\,\rE[p_{r}(\xi)]\,\rE[p_{t}(\xi)\xi] \\
&\quad + t\,\rE[p_{t}(\xi)]\, \rE[p_{r}(\xi)\xi] - rt\,\rE[p_{r}(\xi)]\,\rE[p_{t}(\xi)]
   + n^{-1}R_2(n,r,t), \nonumber
\end{align}
and that $R_2(n,r,t)$ satisfies~\eqref{D3}.

To this end, observe that
\begin{align*}
 {n \choose r,t,n-r-t}
\Big(\rE\!\Big[q_X^r &q_Y^t (1-q_X-q_Y)^{\,n-r-t}\Big]
 - n^{-1}\alpha \rE\!\left[q_X^{r+t}(1-2q_X)^{\,n-r-t}\right]\Big)\\
 &=\Big(1-(2n)^{-1}(r+t)(r+t-1)+n^{-2}\varphi^*_n(r+t)\Big)Q_n(r,t)\\
 &-n^{-1}\alpha (r!t!)^{-1}(1-n^{-1}\varphi_n(r+t))\rE\Big[\xi^{r+t}(1-2n^{-1}\xi)^{\,n-r-t}\Big].
\end{align*}
Representing the right hand side as $I_0+n^{-1}I_1+n^{-2}I_2$, we find by Lemma \ref{Co2},
\begin{align*}
I_0&=\rE[p_r(\xi)]\rE[p_t(\xi)], \\
I_{1}&=(r+t)\big(\rE[p_r(\xi)]\rE[p_t(\xi)\xi]+\rE[p_t(\xi)]\rE[p_r(\xi)\xi]\big) \nonumber\\
&-2^{-1}\rE[p_r(\xi)]\rE[p_t(\xi)\xi^2]-2^{-1}\rE[p_t(\xi)]\rE[p_r(\xi)\xi^2]-\rE[p_r(\xi)\xi]\rE[p_t(\xi)\xi]\\
&-2^{-1}(r+t)(r+t-1)\rE[p_r(\xi)]\rE[p_t(\xi)]-\alpha \rE[p_r(\xi)p_t(\xi)],\\
I_{2}&=(r!t!)^{-1}\rE\Big[\xi^r\eta^t \delta_n^*(r+t,\xi+\eta)\Big]\\
&-2^{-1}(r+t)(r+t-1)(r! t!)^{-1}\rE\Big[\xi^r\eta^t\delta_n(r+t,\xi+\eta)\Big]+\varphi^*_n(r+t)Q_n(r,t)\\
&-\alpha(r!t!)^{-1}\Big(\varphi_n(r+t)\rE\Big[\xi^{r+t}(1-2n^{-1}\xi)^{\,n-r-t}\Big]-\rE\Big[\xi^{r+t}\delta_n(r+t, 2\xi)\Big]\Big).
\end{align*}


On the other hand,  by \eqref{R1} and Lemma \ref{Ld},
\begin{align*}
\rE[\hat q_r]\rE[\hat q_t]&=\Big(\rE[p_{r}(\xi)]+(2n)^{-1}\rE[p_{r}(\xi)(r-(\xi-r)^2)]+n^{-2}R_1(n,r)\Big)\rE[\hat q_t]\\
&=\rE[p_{r}(\xi)]\Big(\rE[p_{t}(\xi)]+(2n)^{-1}\rE[p_{t}(\xi)(t-(\xi-t)^2)]+n^{-2}R_1(n,t)\Big)\\
&+(2n)^{-1}\rE[p_{r}(\xi)(r-(\xi-r)^2)]\Big(\rE[p_{t}(\xi)]+n^{-1}R_0(n,t)\Big)+n^{-2}R_1(n,r)\rE[\hat q_t],
\end{align*}
so that
\begin{align*}
\rE[\hat q_r]\rE[\hat q_t]&=I_0+n^{-1}I_3+n^{-2}I_4, 
\end{align*}
where
\begin{align*}
I_{3}&=2^{-1}\Big(\rE[p_{r}(\xi)]\rE[p_{t}(\xi)(t-(\xi-t)^2)]+\rE[p_{t}(\xi)]\rE[p_{r}(\xi)(r-(\xi-r)^2)]\Big), \\
I_{4}&=\rE[p_{r}(\xi)]R_1(n,t)+\rE[\hat q_t]R_1(n,r)+2^{-1}\rE[p_{r}(\xi)(r-(\xi-r)^2)]R_0(n,t).
\end{align*}

%

Combining all the terms, we arrive at \eqref{sist} with
\begin{align*}
R_2(n,r,t)&=\varphi^*_n(r+t)Q_n(r,t)-\alpha (r!t!)^{-1}\varphi_n(r+t))\rE\Big[\xi^{r+t}(1-2n^{-1}\xi)^{\,n-r-t}\Big]\\
&+\alpha (r!t!)^{-1}\rE\Big[\xi^{r+t}\delta_n(r+t, 2\xi)\Big]-2^{-1}(r+t)(r+t-1)(r! t!)^{-1}\rE\Big[\xi^r\eta^t\delta_n(r+t,\xi+\eta)\Big]\\
&+ (r!t!)^{-1}\rE\Big[\xi^r\eta^t \delta_n^*(r+t,\xi+\eta)\Big]-\rE[p_{r}(\xi)]R_1(n,t)-\rE[\hat q_t]R_1(n,r)\\
&-r\rE[p_{r}(\xi)(2^{-1}+\xi)]R_0(n,t)+2^{-1}\rE[p_{r}(\xi)(\xi^2+r^2)]R_0(n,t).
\end{align*}
From here we derive \eqref{D3} using previously obtained two-sided bounds as well as
\[0\le\xi,\eta\le\beta,\qquad 0\le \rE[\hat q_r], \rE[p_r(\xi)]\le \frac{\beta^r}{r!},\qquad 0\le Q_n(r,t)\le \frac{\beta^{r+t}}{r!t!}.\]

\subsection*{Proof of (\ref{R2})}

Relation (\ref{R2})  follows from Lemma~\ref{tm0}, together with relation \eqref{sist} with $r=t$. This completes the proof of Theorem~\ref{tm1}.


\begin{thebibliography}{9}

\bibitem{BHJ}
A.~D.~Barbour, L.~Holst, and S.~Janson.
\newblock \emph{Poisson Approximation}.
\newblock Oxford University Press, 1992.


\bibitem{F1}
W. Feller.
\newblock An Introduction to Probability Theory and Its Applications, Volume 1.
 \newblock John Wiley, 1950

\bibitem{GJM}
A. Gnedin, S. Janson, and Y. Malinovsky.
\newblock Maximal Counts in the Stopped Occupancy Problem.
 \newblock \emph{Preprint}	arXiv:2506.20411, 2025


\bibitem{KSC}
V.~F.~Kolchin, B.~A.~Sevast'yanov, and V.~P.~Chistyakov.
\newblock \emph{Random Allocations}.
\newblock Winston \& Sons, Washington, DC, 1978.


\end{thebibliography}
\end{document}